\documentclass[11pt]{amsart}

\usepackage[numeric]{amsrefs}
\usepackage{color}
\usepackage{pst-3dplot}

\numberwithin{equation}{section}
\parskip=\smallskipamount

\newtheorem{theorem}{Theorem}[section]
\newtheorem{lemma}[theorem]{Lemma}
\newtheorem{corollary}[theorem]{Corollary}
\newtheorem{proposition}[theorem]{Proposition}

\theoremstyle{definition}
\newtheorem{definition}[theorem]{Definition}

\newtheorem{notation}[theorem]{Notation}

\newtheorem{remark}[theorem]{Remark}

\newcommand{\C}{\mathbb{C}}

\newcommand{\N}{\mathbb{N}}

\newcommand{\R}{\mathbb{R}}

\newcommand{\Aut}{\mathop{{\rm Aut}}}
\newcommand{\id}{\mathop{{\rm id}}}
\newcommand{\hartogs}[1]{ \mathfrak{h}^{#1} } 
\newcommand{\hartogsh}[1]{ {\widetilde{\mathfrak{h}}}^{#1} } 

\begin{document}
\author{Rafael B. Andrist}
\author{Nikolay Shcherbina}
\author{Erlend F. Wold}
\title[Extension of vector bundles and sprays]{The Hartogs Extension Theorem for holomorphic vector bundles and sprays}
\date{\today}
\begin{abstract}
We study ellipticity properties of complements of compact subsets of Stein manifolds.  
\end{abstract}

\subjclass[2010]{Primary 32D15, 32L05, Secondary 32E10}
\keywords{Extension of vector bundles, Hartogs phenomena, Oka theory}

\maketitle

\section{Introduction}
The main purpose of the present article is to investigate ellipticity properties of complements of compact sets in Stein manifolds, and our main result is the following.  

\begin{theorem}\label{thm2}
Let $X$ be a Stein manifold with $\dim X \geq 3$ and let $K \subset X$ be a compact subset.
\begin{itemize}
\item If $X \setminus K$ is elliptic, then $K$ has only finitely many accumulation points.
\item If $X \setminus K$ is subelliptic and smoothly bounded, then $K$ is empty.
\end{itemize} 
\end{theorem}

For the notion of ellipticity in the sense of Gromov \cites{Gromov-elliptic} as well as the notion of subellipticity we refer to the book of Forstneri\v{c} \cites{Forstneric-book}. A short overview can be found in this article in Section \ref{sec-oka}.

The main ingredient in the proof of Theorem \ref{thm2} is the following Hartogs type extension theorem for holomorphic vector bundles on Stein manifolds which was stated by Siu in \cite{Siu1971}:

\begin{theorem}\label{thm1}
Let $X$ be a Stein manifold and let $K\subset X$ be a holomorphically convex compact set with connected complement.  
Let $E \to X\setminus K$ be a holomorphic vector bundle.
If $\dim X \geq 3$, then there is a finite set of points $P \subset K$ such that $E$ extends to a holomorphic vector bundle on $X \setminus P$.
\end{theorem}

For proving this result Siu suggested to use methods from Andreotti--Grauert theory. To our knowledge no proof of this statement exists in the literature, and so we carry it out in detail. It turns out that the extension of vector bundles across critical level sets of an exhaustion function poses some technical difficulties which may be of independent interest, see Theorem \ref{thm3}.

A proof of Theorem \ref{thm1} in the case of line bundles was recently given by Forn\ae ss--Sibony--Wold \cite{FornaessSibonyWold}, and a description of the obstructions for the extension of roots of line bundles in dimension two was given by Ivashkovich \cite{Ivashkovich2}.

\begin{theorem}\label{thm3}
Let $X$ be a complex manifold with $\dim X \geq 3$ and let $M \subset X$ be a closed totally real Lipschitz subset of $X$.
Then for any holomorphic vector bundle $E \to X \setminus M$, there exists a discrete set of points $P \subset M$ such that $E$ extends to a holomorphic vector bundle on $X \setminus P$.
\end{theorem}

In view of Theorem \ref{thm3} we observe that it follows from the proof of Theorem \ref{thm1} when $X$ is a complex manifold, $M\subset X$ is a closed totally real Lipschitz set, and if $X\setminus M$ is elliptic, then $M$ is a discrete set of points.  

The paper is organised as follows. In Section \ref{concave} we present some results concerning the density of pseudoconcave points and the extension of mappings. Next, in Section \ref{fibrebundles}, we consider extensions of sections and maps of fibre bundles, which -- among other methods -- are used to show the uniqueness of the extensions of bundles.
In Section \ref{smoothlevelsets} we show how to extend bundles across non-critical level sets, and in Section \ref{totallyrealsets} we extend bundles across totally real Lipschitz sets, proving Theorem \ref{thm2}. In Section \ref{modification} we show how to modify exhaustion functions near critical points and we give the proof of Theorem \ref{thm1}. Finally, in section \ref{sec-oka}, we explain the application to Oka theory and prove Theorem \ref{thm2}.

\section{Pseudoconcave points of compact sets and extension of mappings}\label{concave}

We denote the coordinates on $\C^n$ by $z_1, \dots, z_n$ with $z_j = x_j + i y_j$, and we let $x$ denote the tuple $(x_1, \dots, x_n)$ and $y$ denote the tuple $(y_1, \dots, y_n)$. 
By a Lip-$\alpha$-graph $M$ at the origin we mean a set $M = \{ y=\psi(x) \}$ with $\psi(0) = 0$ and $\psi$ Lipschitz continuous with Lipschitz constant $\alpha>0$.

\begin{definition}
For $0 \leq r < s$ define according to Siu \cite{Siu}*{chap 2. sec. 3}
\begin{align*}
G^n(r,s) &:= \left\{ (z_1, \dots, z_n) \in (s \cdot \triangle)^n \,:\, |z_k| > r \text{ for some } k \in \{1, \dots, n\} \right\} \\
&= s \cdot \triangle^n \setminus r \cdot \overline{\triangle}^n.
\end{align*}

For $q = 1, 2, \dots, n-1$ we define a \emph{standard $q$-Hartogs figure} $\hartogs{q}$ by
\[
\hartogs{q} := (\triangle^{q}\times G^{n-q}(1/2,1)) \cup (1/2 \cdot \triangle^{q}\times \triangle^{n-q}),
\] 
and by a \emph{$q$-Hartogs figure}, with some abuse of notations also denoted by $\hartogs{q}$, we mean the biholomorphic image $\phi(\hartogs{q})$ of a standard $q$-Hartogs figure under an injective holomorphic map $\phi \colon \triangle^n \to X$ into a complex manifold $X$.
We set $\hartogsh{q} := \phi(\triangle^n)$. If not stated explicitly otherwise, we always assume $n = \dim X$.
\end{definition}

\begin{definition}
Let $K \subset \C^n$ be a compact set.
We call a point $x \in K$ \emph{strictly $q$-pseudoconcave} if for any open neighborhood $U$ of $x$, there exists a $q$-Hartogs figure $\hartogs{q} \subset U \setminus K$ with $x \in \hartogsh{q}$.
\end{definition}

We will use the following notation: For any set $K \subset \C^k \times \C^m$ and $z_0 \in \C^k$, we denote
\[
K_{z_0} := K \cap (\{z_0\} \times \C^m)
.\]
Note that if $x \in K_{z_0}$ is a strictly $q$-pseudoconcave point for $K_{z_0}$ in $\mathbb C^m$, then $x$ is strictly $(q+k)$-pseudoconcave for $K$, since $q$-Hartogs figures in $\{z_0\} \times \C^m$ can be fattened up slightly.  

\begin{proposition}\label{HPC}
Let $K \subset \C^n$ be a compact set.   
Then the strictly $q$-pseudoconcave points in $K$ are dense in the Shilov boundary $\check{S}(P(K))$ of the algebra $P(K)$ for $q = 1, 2, \dots, n-1$.
\end{proposition}

\begin{proof}
Note first that a global extreme point $ x\in K$ is $q$-pseudoconcave for 
$0<q<n$: 
Let $x \in K$ be such that $\|x\|= \sup\limits_{z\in K} \| z \|$. After a unitary transformation we may assume that $x = e_1= (1,0,\dots,0)$. Then clearly $(\{e_1\} \times \{\C^{n-1}\}) \cap K = \{e_1\}$, so it is easy to construct $q$-Hartogs figures at $x$.

\medskip
We proceed to show that if $x \in \check{S}(P(K))$ and if $\epsilon > 0$, then there exists a point $x' \in K$ with $\|x'-x\| < \epsilon$ and a holomorphic automorphism $\alpha \in \Aut_{\mathrm{hol}}(\C^n)$ such that $\alpha(x')$ is a global extreme point for the set $\alpha(K)$.

The peak points for $P(K)$ are dense in $\check{S}(P(K))$, so we may assume that the point $x$ is a peak point for the algebra $P(K)$.  
Then $x$ is not in the polynomially convex hull of the set $K' := K \setminus B_\epsilon(x)$. Hence, for any $R>0$ and any $\delta>0$, 
by the Anders\'{e}n--Lempert Theorem \citelist{\cite{ForstnericRosay} \cite{Andersen} \cite{AndersenLempert}} there exists $\Lambda \in \Aut_{\mathrm{hol}}(\C^n)$ such that $\|\Lambda - \id\|_{K\setminus B_\epsilon(x)} < \delta$ and such that $\Lambda(x) > R$. 
If $R$ is chosen large and $\delta$ is chosen small, this proves the proposition.  
\end{proof}

\begin{corollary}\label{densehpc}
Let $M \subseteq \C^n$ be a closed subset of a closed Lip-$\alpha$ graph near the origin with $0<\alpha<1$.  
Then there exists a dense set of points $\Sigma\subset M$ such that for any point $x\in\Sigma$ and any open set $U$ containing $x$, there exists an embedding $\phi \colon \triangle^n \to U$ containing $x$, with $\phi(\hartogs{1})\cap M = \emptyset$.
\end{corollary}
\begin{proof}
Lip-$\alpha$ graphs over $\mathbb R^n$ are polynomially convex if $\alpha<1$ \cite{Stout}*{Cor. 1.6.11, p. 56}. Furthermore $M=\check{S}(P(M))$.
\end{proof}

\begin{corollary}\label{lipremovable}
Let $X$ be a complex manifold with $\dim X \geq 2$ and let $M \subset X$ be a closed totally real Lipschitz submanifold.
Furthermore let $Z$ be a complex manifold, let $Y$ be a Stein manifold and let $f \colon (X\setminus M)\times Z \to Y$ be a holomorphic 
map.
Then $f$ extends holomorphically across $M \times Z$.  
\end{corollary}
\begin{proof}
The problem is local, so let $W$ be an open connected subset of the origin, assume that $M$ is a closed subset of $W$, and let $f\in\mathcal O((W\setminus M)\times Z)$.
Since $Y$ is Stein, it is enough to consider functions.  
Note that $W\setminus M$ is connected.   
By the previous corollary there exists a point $x\in M$ and an embedding $\phi \colon \triangle^n \to W$ with $x\in\phi(\triangle^n)$ and $\phi(\hartogs{1})\cap M=\emptyset$.
Note that $\phi(\triangle^n)\setminus M$ is connected.    Let $\tilde f$ be the extension of $f|_{\phi(\hartogs{1})\times Z}$ to $\phi(\triangle)\times Z$ according to Lemma \ref{localextension}.
Now $f$ and $\tilde f$ are both holomorphic on the connected open set $\phi(\triangle^n)\setminus M$ and they coincide on the open set $\phi(\hartogs{1})$.
By the identity principle they coincide on $\phi(\triangle)\setminus M$, and then clearly $\tilde f$ extends $f$ across $(M\cap\phi(\triangle^n))\times Z$.

Now let $\Omega\subset M$ be the largest open set (in the relative topology) such that $f$ extends across $\Omega\times Z$. This is well defined since any extension across a point of $M$ is unique if it exists, and it is non-empty by the argument above.  
Note that $M\setminus\Omega$ is either empty or it is a closed subset of a Lip-$\alpha$-graph, and in the latter case we would get a contradiction to the assumption that $\Omega$ is maximal, since, due to Corollary \ref{densehpc}, we could apply the above argument again. Hence, $M = \Omega$.
\end{proof}

\begin{lemma}\label{localextension}
Let $Z$ be a complex manifold and let $f \colon \hartogs{q}\times Z \to \mathbb C$ be a holomorphic function. Then $f$ extends to a holomorphic 
function $\tilde f \colon \triangle^n\times Z \to \C$.
\end{lemma}
\begin{proof}
Let $(w,z)$ denote the coordinates on $\triangle^n\times Z$. By the classical Hartogs extension theorem, for each $z\in Z$ the map $f(\cdot,z)$ extends uniquely to a holomorphic function $\tilde f_z \colon \triangle^n \to \C$, and we let these extensions determine an extension $\tilde f$.
It remains to show that the extension is actually holomorphic, and since holomorphicity is a local property, we can use local coordinates and assume that $Z$ is a domain in $\C^m$.
Using the classical Hartogs extension theorem for partial derivatives in the $z$-directions on $\hartogs{q} \times \mathbb C^m$, it is easy to see that $\tilde f$ is holomorphic in each variable separately, hence $\tilde f$ is holomorphic.  
\end{proof}

If $M$ is $\mathcal{C}^2$-smooth, it is easy to construct a $1$-Hartogs figure directly:

\begin{lemma}
\label{smoothgraph}
Let $M = \{ y = \psi(x) \}$ be a $\mathcal C^2$-smooth graph at the origin with $d\psi(0) = 0$. Then for any open set $U$ containing the origin, there exists an embedding $\phi \colon \triangle^{n} \to U$ containing the origin, with $\phi(\hartogs{1})\cap M=\emptyset$.
\end{lemma}

\begin{figure}[ht]
\begin{center}
\begin{pspicture}(-3,-1.5)(3,4)
\pstThreeDCoor[xMin=-4, xMax=4, yMin=-4, yMax=4, zMax=3, arrows=->, nameX=$\R^{n-1}_y$, nameY=$\R^n_x$, nameZ=$y_n$, linecolor=black]
\psplotThreeD[drawStyle=xLines, plotstyle=line, linecolor=black, yPlotpoints=40, xPlotpoints=30, linewidth=0.2pt, hiddenLine=true](0,3)(-1.5,1.5){y dup mul}
\psplotThreeD[drawStyle=yLines, plotstyle=line, linecolor=black, yPlotpoints=40, xPlotpoints=30, linewidth=0.2pt, hiddenLine=true](0,3)(-1.5,1.5){y dup mul}
\end{pspicture}
\end{center}
\caption{\label{img-easyhartogs} The lower bound $y_n = (C/2 - K) \cdot \|x\|^2$ for $\mathrm{Im}( f_C(x, \psi(x)) )$ in Lemma \ref{smoothgraph}.}

\end{figure}
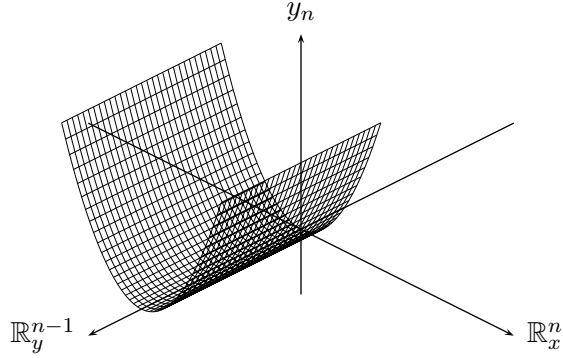

\begin{proof}
For $C>0$ we let $F_C \colon \C^n \to \C^n$ be the map 
\[
F_C(z)=(z_1,\dots,z_{n-1},f_C(z)) := \left(z_1,\dots,z_{n-1},z_n + C i \cdot \sum_{j=1}^{n} z_j^2 \right).
\]
Since $F_C$ is injective near the origin, there exists a constant $K>0$ such that $\|\psi(x)\| < K \|x\|^2$, and by staying close enough to the origin we may assume that $\|\psi(x)\| \leq 1/2 \cdot \|x\|$. 
Choosing $C > 2K$ we see that 
\begin{align*}
\mathrm{Im}(f_C(x,\psi(x))) & =    C \sum_{j=1}^n\left( x_j^2 - \psi_j(x)^2 \right) + \psi_n(x)\\
                            & \geq C \sum_{j=1}^n 1/2 \cdot x_j^2 - K\|x\|^2, 
\end{align*}
Therefore we obtain $\mathrm{Im}( f_C(x, \psi(x)) ) \geq (C/2 - K) \cdot \|x\|^2$.
Using this it is easy to construct a Hartogs figure.
\end{proof}

\section{The Hartogs extension theorem for holomorphic fibre bundles}\label{fibrebundles}

\begin{notation}
We denote the super-level sets of an exhaustion function $\rho$ of the complex manifold $X$ by $X^c := \{ x \in X \,:\, \rho(x) > c \}$, the sub-level sets by $X_c := \{ x \in X \,:\, \rho(x) < c \}$, and its level sets by $\Gamma_c := \{ x \in X \,:\, \rho(x) = c \}$.
\end{notation}

\begin{definition}
\label{defnice}
Let $\rho \colon X \to \R$ be a Morse exhaustion function of the complex manifold $X$. A critical point $x_0 \in X$ of $\rho$ is called \emph{nice} (see \cites{HenkinLeiterer}) if there exist local holomorphic coordinates $z = (z_1, \dots, z_n) = (x_1 + i y_n, \dots, x_n + i y_n)$ such that $\rho$ is of the form 
\begin{equation}
\label{nice}
\rho(z) = \rho(x_0) + \sum_{j=1}^{m} \left( x_j^2 + \mu_j y_j^2 \right) + \sum_{j=m+1}^n \left( x_j^2 - \mu_j y_j^2 \right),
\end{equation}
with $0 \leq \mu_j \leq 1$ for $j = 1, \dots, n$ and $\mu_j < 1$ for $j = m+1, \dots, n$.

For a compact set $L \subset X$ the Morse exhaustion function is called \emph{nice} on $X \setminus L$ if the set of critical points of $\rho$ in $X \setminus L$ is discrete and and if each of them is nice, with not more than one critical point in each level set.
\end{definition}

\begin{theorem}\label{extensionsectionandisom}
Let $X$ be a Stein manifold, $\dim X\geq 2$, let $K\subset K'\subset X$ be compact sets with $K$ holomorphically convex and $X\setminus K'$ connected, and let $Q\subset X\setminus K$ be a closed discrete set of points.
Let $\pi \colon E \to X\setminus (K\cup Q)$ be a holomorphic fibre bundle with Stein fibre, and let $Y$ be a Stein manifold.
Then the following hold
\begin{itemize}
\item[(i)] Any holomorphic map
$f \colon \pi^{-1}(X\setminus (K'\cup Q)) \to Y$ extends uniquely to a holomorphic map
$\tilde f \colon E \to Y$, and 
\item[(ii)] Any holomorhic section $s \colon X\setminus (K'\cup Q) \to E$ extends uniquely to a section $\tilde s \colon X\setminus (K\cup Q) \to E$
\end{itemize}
\end{theorem}

\begin{proof}
We first argue that we can assume that $Q=\emptyset$.
For this purpose we let $\rho$ be a strictly plurisubharmonic exhaustion function of $X$, and we pick $c>0$ such that $K'\subset\{\rho<c\}$ and $\{\rho=c\}\cap Q=\emptyset$.
Choose $c'<c$ such that $K'\cup (Q\cap\{\rho<c\})\subset\{\rho<c'\}$.
Then $K\cup (Q\cap\{\rho<c\})$ is holomorphically convex in $\{\rho<c\}$, so we can let this be our new set $K$, and $K'$ be the set $\{\rho\leq c'\}$, and we replace $X$ by $\{\rho < c\}$.

\medskip

Let $\rho \colon X \to \R$ be a $\mathcal C^\infty$-smooth non-negative plurisubharmonic exhaustion function such that $K = \{ \rho=0 \}$, $\rho>0$ on $X\setminus K$, and $\rho$ is strictly plurisubharmonic on $X \setminus K$ (see e.g \cite{Stout}*{Theorem 1.3.8}).
By the Morse Lemma and \cites{HenkinLeiterer} we can assume that $\rho$ is a nice exhaustion function, i.e.\ near each critical point $x_0 \in X \setminus K$ the exhaustion function $\rho$ is of the form \eqref{nice} in suitable local coordinates.
We will give the proof of (i); the proof of (ii) is almost identical.  

\medskip

Choose a non-critical value $c>0$ such that $K' \subset X_c$, and define 
\begin{equation}
s = \inf \{0\leq t\leq c \,:\, f \mbox{ extends to } \pi^{-1}(X^t)\}.
\end{equation}
By Corollary \ref{noncriticalext} below we obtain that $s<c$. Assume to get a contradiction that $s>0$.
By Corollary \ref{noncriticalext} or Lemma \ref{criticalext} $f$ would extend further.  
\end{proof}

We now give the lemmas used in the above proof.  
   
\begin{lemma}
Let $c'>0$ be a non-critical value for $\rho$, and assume that $g \colon \pi^{-1}(X^{c'}) \to Y$ is a holomorphic map. Then there exists $c''<c'$  and an extension of $g$ to $\pi^{-1}(X^{c''})$.
\end{lemma}   

\begin{proof}
For any point $x \in \Gamma_{c'}$ there exist local holomorphic coordinates such that $\Gamma_{c'}$ is strictly convex near $x$.
Hence, by compactness, there exist Hartogs figures
$(\hartogs{1}_j, \hartogsh{1}_j), \; j=1, \dots, m$, in $X^{c'}$ such that for all $j$  
\begin{equation}\label{coverhartogs}
\hartogsh{1}_j \cap \tilde X^{c'} \mbox{ is connected for all small $\mathcal C^2$-perturbations $\tilde X^{c'}$ of $X^{c'}$}.
\end{equation}
and such that $\cup_j \hartogsh{1}_j$ covers $\Gamma_{c'}$.
Choose finitely many compact sets $C_j \subset \Gamma_{c'} \cap \hartogsh{1}_j$ such that $\cup_j C_j$ covers $\Gamma_{c'}$.   
Since $Y$ is Stein, we can think of each $g|_{\pi^{-1}(\hartogs{1}_j)}$ as a holomorphic map to $\mathbb C^k$ 
and so a local extension reduces to extension of each component.  Hence Lemma \ref{localextension}
gives an extension to each $\pi^{-1}(\hartogsh{1}_j)$.  It is intuitively clear that we get a jointly well defined extension across $\Gamma_{c'}$.  The following is a precise argument.

We proceed by induction on $l$ and assume that we have found a small $\mathcal C^2$-perturbation $\rho_{l}$ of $\rho$ defining a small perturbation $X^{c'}_{l}$ of $X^{c'}$, with $X^{c'} \cup C_1 \cup \cdots \cup C_{l} \subset X^{c'}_{l}$ and an extension of $f$ to $X^{c'}_{l}$.

By \eqref{coverhartogs} we get a well defined extension of $f$ to $X^{c'}_{l}\cup\hartogsh{1}_{l+1}$.
Let $\chi \in \mathcal C^{\infty}_0(X), \chi\geq 0, \mathrm{supp}\, \chi \subset \hartogsh{1}_{l+1}$ and $\chi\equiv 1$ near $C_{{l}+1}$. Then $\rho_{{l},\epsilon} := \rho_{l} + \epsilon \cdot \chi$ converges to $\rho_{l}$ in $\mathcal C^2$-norm, and defines $\rho_{l+1}$ and  $X^{c'}_{{l}+1}$ for $\epsilon>0$ small enough.

After $m$ steps we have an extension to a full neighborhood of $\Gamma_{c'}$.
\end{proof}

\begin{corollary}\label{noncriticalext}
Let $c'>0$ be a non-critical value for $\rho$, and $g \colon \pi^{-1}(X^{c'}) \to Y$ be a holomorphic map.
Let $0 \leq c'' < c'$ and assume that there are no critical values in the interval $(c'',c')$.
Then $g$ extends to a holomorphic map $\tilde g \colon \pi^{-1}(X^{c''}) \to Y$.
\end{corollary}
\begin{proof}
We let $s:=\inf\{c''<t<c' \,:\, g \mbox{ extends to } \pi^{-1}(X^{t})\}$.
Then $c''\leq s<c'$ and $g$ extends to some $\tilde g \colon \pi^{-1}(X^{s}) \to Y$.
If $c''<s$ we would get a contradiction since $\tilde g$ would extend across $\Gamma_s$.
\end{proof}

Next we consider the critical case.  

\begin{lemma}\label{criticalext}
Let $c'>0$ be a critical value for $\rho$ and
$g \colon \pi^{-1}(X^{c'}) \to Y$ be any holomorphic map.
Then there exists a $c'' \in [0,c')$ such that $g$ extends 
to $\pi^{-1}(X^{c''})$.
\end{lemma}

\begin{proof}
If $c'$ corresponds to a local minimum, it is easy to extend using a Hartogs figure and Lemma \ref{localextension}.
Otherwise we choose $c>c'$ such that there are no critical values in the interval $(c',c)$.
Let $\tilde\rho$ be a function as in Proposition \ref{lemperturbflat} below.  
Then, by Corollary \ref{noncriticalext} with the function $\tilde\rho$, the map $g$ extends to $\pi^{-1}(\{\tilde\rho>c'\})$, and, by Lemma \ref{lipremovable}, $g$ extends further across the critical point.
So there exists some $c''<c'$ such that the extension is defined on $\pi^{-1}(X^{c''})$.
\end{proof}

\begin{remark}
\label{rembundlesect}
Note that an isomorphism between two holomorphic vector bundles of rank $r$ is a section of an associated holomorphic fibre bundle with fibers $\mathrm{Gl}_r(\C)$.  Hence part (ii) of Theorem \ref{extensionsectionandisom} can be applied for extension of isomorphisms of bundles.  
\end{remark}

Before we give two corollaries, we state the definition of the extension of a bundle. 
\begin{definition}
Let $X$ be a complex manifold, let $U_1\subset U_2\subset X$ be open sets and let $E_j\rightarrow U_j$ be holomorphic
vector bundles for $j=1,2$.   We say that $E_2$ is an \emph{extension} of $E_1$ if $E_2|_{U_1}$ is isomorphic 
to $E_1$.
\end{definition}

\begin{corollary}
Let $X$ be a Stein manifold with a nice exhaustion function $\rho$. Let $c_3<c_2<c_1$, let $Q_j$ be closed discrete subsets of $X^{c_j}$ for $j=1,2,3$, and let $E_1 \to X^{c_1} \setminus Q_1$ be a holomorphic vector bundle, and assume that $E_j$ is an extension of $E_1$ to $X^{c_j}\setminus Q_j$ for $j=2,3$.
Moreover, assume for each $j=1,2,3$ that $Q_j$ is the singularity set of the vector bundle $E_j \to X^{c_j} \setminus Q_j$, i.e. $E_j$ cannot be extended as a holomorphic vector bundle through any point in $Q_j$.  
Then $Q_1\subset Q_2\subset Q_3$, $Q_1=Q_2\cap X^{c_1}=Q_3\cap X^{c_1}$, $Q_2 = Q_3 \cap X^{c_2}$ and $E_2$ and $E_3$ are isomorphic over $X^{c_2}\setminus Q_2$.
\end{corollary}

\begin{proof}
By definition of an extension we have that $Q_j\cap X^{c_1}\subset Q_1$ for $j=2,3$,
and by transitivity, the three bundles are isomorphic over $X^{c_1} \setminus Q_1$. 
Now let $x\in X^{c_1}\setminus Q_j$ for $j=2$ or $j=3$.
The isomorphism implies that $E_1$ is trivial in a punctured neighbourhood of $x$, hence $x\notin Q_1$.
Hence $Q_1\subset Q_2$ and $Q_1\subset Q_3$.
It follows that $Q_1=Q_2\cap X^{c_1}=Q_3\cap X^c_{1}$.
Using a similar argument, we conclude that $Q_2 \subset Q_3$ and $Q_2 = Q_3 \cap X^{c_2}$.

Finally, we observe that by Theorem \ref{extensionsectionandisom} the isomorphism between $E_2$ and $E_3$ on $X^{c_1}\setminus Q_1$ 
extends to an isomorphism on $X^{c_2}\setminus Q_2$.
\end{proof}

\begin{corollary}\label{inductivelimit}
Let $X$ be a Stein manifold with a nice exhaustion function $\rho$. 
Let $c_j \in \R$, $c_j<c_{j-1}$, for $j=1,2,3,\dots$. Let $E_1 \to X^{c_1}\setminus Q_1$ be holomorphic vector bundle, and assume that for every $j=1,2,3,\dots$ there exist holomorphic extensions $E_j \to X^{c_j} \setminus Q_j$, 
where $Q_j$ is the singularity set of $E_j$.
Then $E_1$ extends to $X^s\setminus Q$ as a holomorphic vector bundle, where $s = \inf \{c_j\}$ and $Q = \bigcup_{j=1}^\infty  Q_k$.
\end{corollary}
\begin{proof}
The previous corollary allows us to define the inductive limit of all the extensions, removing unnecessary singular points.   
\end{proof}

\begin{proposition}\label{extensionsection}
Let $K \subset \triangle \times r \cdot \overline{\triangle}^{n-1}, 0 < r < 1, n \geq 3$, be a closed set with $\triangle^n \setminus K$ connected, and with $K_{z_1}$ polynomially convex (possibly empty for some $z_1$) for all $z_1 \in \triangle$.
Let $\pi \colon E \to \triangle^{n}\setminus K$ be a holomorphic fibre bundle with Stein fibre, and let $Y$ be a Stein manifold.
If we denote by ${\mathfrak{D}}^{n-1}_{r}$ the domain $\triangle \times G^{n-1}(r,1) = \triangle \times (\triangle^{n-1} \setminus r \cdot \overline{\triangle}^{n-1})$, then the following hold

\begin{itemize}
\item[(i)] Any holomorphic map
$f \colon \pi^{-1}({\mathfrak{D}}^{n-1}_{r}) \to Y$ extends uniquely to a holomorphic map
$\tilde f \colon E \to Y$, and 
\item[(ii)] any holomorhic section $s \colon {\mathfrak{D}}^{n-1}_{r} \to E$ extends uniquely to a section $\tilde s \colon \triangle^n\setminus K \to E$.
\end{itemize}
\end{proposition}

\begin{proof}
Theorem \ref{extensionsectionandisom} implies that $f$ and $s$ have well defined extensions to the slices $\triangle^{n}_{z_1} \setminus K_{z_1}$ for all $z_1 \in \triangle$. 
This gives us extensions $\tilde f$ and  $\tilde s$ of $f$ and $s$ respectively, but we need to show that they are holomorphic in the $z_1$-direction.  

In order to show this, say for $s$, we let $\Omega$ be the largest open subset of $\triangle^n \setminus K$ such that $s$ is holomorphic on $\Omega$, and write $\tilde K=\triangle^n\setminus\Omega$.    
Assume by contradiction that $\tilde K\neq K$.
Then there exists $z_0\in\triangle$ such that $K_{z_0}$ is a proper subset of $\tilde K_{z_0}$. Since $K_{z_0}$ is polynomially convex we see that $K_{z_0}$ cannot contain 
$\check{S}(P(\tilde K_{z_0}))$, hence there is a point $x\in\check{S}(P(\tilde K_{z_0}))\setminus K_{z_0}$. It follows then from Proposition \ref{HPC} that there is a strongly pseudoconcave point $x'\in\tilde K_{z_0}\setminus K_{z_0}$, and hence there exists an arbitrarily small Hartogs figure $(\hartogs{1}, \hartogsh{1})$ of dimension $(n-1)$ with $\hartogs{1} \subset \triangle^n_{z_0} \setminus \tilde K_{z_0}$ and $x' \in \hartogsh{1}_{z_0}$.
For some small $\epsilon>0$ we have that $\triangle_{\epsilon}(z_0)\times \hartogs{1}\subset\triangle^n\setminus\tilde K$, and this gives us a Hartogs figure on which $s$, by local triviality of the fibre bundle $\pi$, extends to some open neighborhood of $x'$ in $\triangle^n$. Contradiction.

The case of a holomorphic map follows the same argument.  
\end{proof}

\section{Extending vector bundles across smooth level sets of a strictly plurisubharmonic exhaustion function}\label{smoothlevelsets}

In this section we will assume throughout that $X$ is a complex manifold with $\dim X = n \geq 3$, and 
we let $\rho$ be a nice exhaustion function.  

\begin{proposition}\label{extendnoncritical}
Let $c$ be a non-critical value of $\rho$, $Q \subset X^c$ be a closed (in $X^c$) discrete set of points and let $E\rightarrow X^c \setminus Q$ be a holomorphic vector bundle whose singular set is $Q$.
Then $Q$ is finite near $\Gamma_c$ and there exist $c'<c$ and a finite set of points $P \subset X^{c'}$, such that $E$ extends to $X^{c'}\setminus (Q\cup P)$.
\end{proposition}

The first step for proving the proposition will be the following lemma:

\begin{lemma}\label{extendsheaf}
Under the assumptions of Proposition \ref{extendnoncritical} the set $Q$ is finite near $\Gamma_c$ and there exists a finite set of points $P \subset \Gamma_{c}$ such that for any point $x\in\Gamma_{c}\setminus P$, there exists an open neighborhood $U$ of $x$ in $X$ with the property that $E$ is a holomorphically trivial vector bundle on $U\cap X_{c}$. Moreover, there exists an open neighborhood $V$ of $P$ such that $E$ extends as a coherent analytic sheaf $\mathcal S$ on $X^{c}\cup V$.
\end{lemma}

\begin{proof}
Let $x \in \Gamma_{c}$. Then we can assume that $\Gamma_{c}$ is strictly convex in suitable local holomorphic coordinates, and hence it is easy to construct a holomorphic embedding $\phi \colon \triangle^n \to X$ with $\phi(0) = x$ and a set $W$ of the form
\begin{equation}
\label{eq-extendsheaf-hartogs}
\left( \triangle_\varepsilon(z^\ast) \times \triangle^{n-1} \right) \cup \left( \triangle \times G^{n-1}(r,1) \right)
\end{equation}
for some $z^\ast \in \triangle$, $\varepsilon \in (0,1)$, $r \in (0,1)$ and $\triangle_\varepsilon(z^\ast) = \{z \, : \, |z - z^\ast| < \varepsilon \}$
with
$\varphi(W) \subset X^c\setminus Q$,
where $Q$ is the singularity set of $E$.

Indeed, due to the strict convexity of $\Gamma_c$, we can first choose an embedding $\varphi \colon \triangle^n \to X$ such that $\varphi(0) = x$ and $\varphi\left( (\{0\} \times \triangle^{n-1}) \cup(\triangle \times b \triangle^{n-1}) \right) \subset X^c \cup \{x\}$.
Then, taking into account the fact that $Q$ is countable, we can choose $z^\ast \in \triangle$ and $t \in (0,1)$ such that $\varphi\left( (\{z^\ast\} \times \triangle_t^{n-1}) \cup(\triangle \times b \triangle_t^{n-1}) \right) \subset X^c \setminus Q$ and hence, by closedness of $Q$ in $X^c$, there exists a small enough neighborhood of the set $(\{z^\ast\} \times \triangle_t^{n-1}) \cup (\triangle \times b \triangle_t^{n-1})$ in $\triangle^n$ which will give, after a homothety $\triangle_t^{n-1} \to \triangle^{n-1}$, the open set $W$ in \eqref{eq-extendsheaf-hartogs} with the required properties.

Then by Siu, page 225 in \cites{Siu}, the vector bundle $E$ extends to a coherent analytic sheaf $\mathcal S$ on $\phi(\triangle^n)$, which is a holomorphic vector bundle $\tilde E \to \phi(\triangle^n) \setminus \tilde Q$ where $\tilde Q$ is a finite set of points.
In view of Remark \ref{rembundlesect} above we can conclude by Proposition \ref{extensionsection}, applied to the fibre bundle with fibres $\mathrm{Gl}_r(\C)$, that any isomorphism $\Psi \colon E|_G \to \tilde E|_G$, where $G := \phi(\triangle \times G^{n-1}(r, 1))$, extends to a vector bundle isomorphism 
\[
\tilde\Psi \colon E|_{\phi(\triangle^n) \setminus (X_{c} \cup Q \cup \tilde Q)} \to \tilde E|_{\phi(\triangle^n) \setminus (X_{c} \cup Q \cup \tilde Q)}.
\]
In particular, the set $Q \cap \varphi(\triangle^n) \subset \widetilde Q$ is finite. By the compactness of $\Gamma_c$ we can find a finite covering $\Gamma_c$ by the sets $\varphi_j(\triangle^n), j = 1,2, \dots, m$, described above. Therefore, the set $P$, defined as the union of the corresponding sets $\widetilde Q_j \cap \Gamma_c$ (where $\widetilde Q_j$ is the described set $\widetilde Q$ which corresponds to $\varphi_j(\triangle^n)$), as well as the set $Q$ near $\Gamma_c$ are finite.
For $x \in \Gamma_c \setminus P$, we can take an embedding $\varphi \colon \triangle^n \to X$ with $x \in \varphi(\triangle^n)$ and such a small image that $\varphi(\triangle^n) \cap Q = \emptyset$. This shows that for $U = \varphi(\triangle^n)$ the bundle $E$ is trivial on $U \cap X_c$.
The existence of the required neighborhood $V$ of $P$ and the extension of the vector bundle $E \to X^c \setminus Q$ to a coherent sheaf $\mathcal S$ on $X^c \cup V$ follows from the above argument applied to a covering of $P$ by finitely many polydiscs $\varphi_j(\triangle^n), \; j = 1,\dots,m$, and then extending $E$ to a coherent sheaf $\mathcal S_j$ for each $j=1,\dots,m$.
\end{proof}

\begin{proof}[Proof of Proposition \ref{extendnoncritical}]
We now let $\chi \in \mathcal C^\infty_0(X)$ be such that $\chi \geq 0$ and $\chi(p) > 0$ for all $p\in P$ provided by the preceding lemma. 
If $\epsilon>0$ is small enough, then the function $\rho_\epsilon := \rho + \epsilon \chi$ is strongly plurisubharmonic without critical points near $\Gamma_{c,\epsilon} := \{x \in X \,:\, \rho_\epsilon(x) = c\}$, and by the last lemma we have an extension of $E$ across $\Gamma_{c, \epsilon}$ which we may assume to be a trivial bundle near any point of $\Gamma_{c,\epsilon}$.
To avoid too many indices we now drop the subscript $\epsilon$, assume that $E$ was already locally trivial near $\Gamma_c$, and proceed with the extension.

Let $U_i, i=1,\dots,l$, be an open cover of $\Gamma_c$ such that $E|_{X^c \cap U_i}$ is trivial for every $i$.
Choose a finite number of Hartogs figures $(\hartogs{1}_j, \hartogsh{1}_j)$ such that for each $j$ one has $\hartogs{1}_j \subset X^c$, $\hartogsh{1}_j \subset U_i$ for some $i$, $\hartogsh{1}_j \cap X^c$ is connected, and such that $\bigcup_j \hartogsh{1}_j$ covers $\Gamma_c$.
It follows from Proposition \ref{extensionsection} and Remark \ref{rembundlesect} that if $\tilde X^c$ is a sufficiently small $\mathcal C^2$-perturbation of $X^c$ with $X^c \subset \tilde X^c$, and if $\tilde E$ is an extension of $E$ to $\tilde X^c$, then $\tilde E$ is holomorphically trivial on $\hartogsh{1}_j \cap \tilde X^c$ for each $j$.
Now, an argument as at the end of the proof of Lemma 3.4 shows that $E$ extends across $\Gamma_c$.
\end{proof}

\begin{corollary}\label{noncriticalextension}
Let $X$ be a Stein manifold with a nice exhaustion function $\rho$, let $Q \subset X^c$ be a closed discrete set for a non-critical value $c$, and let $E\rightarrow X^c\setminus Q$ be a holomorphic vector bundle such that $Q$ is its singularity set.
Then for each $c'<c$ with no critical values in $(c',c]$ there exists a discrete subset  $Q' \subset X^{c'}$ and an extension of $E$ to $X^{c'} \setminus Q'$ with $Q'$ being singularity set of this extension.
\end{corollary}
\begin{proof}
Let $s := \inf \{ c'' < c \;:\; E \mbox{ extends to a vector bundle on } X^{c''}\setminus Q'', \\ \mbox{ where } Q'' \subset  X^{c''} \mbox{ is a closed discrete set} \}$.
By Corollary \ref{inductivelimit} we have that $E$ extends to $X^s \setminus Q^s$ for some closed discrete set $Q^s \subset X^s$ and so $s=c'$.
\end{proof}

\section{Proof of Theorem \ref{thm3}}
\label{totallyrealsets}

Before proving the theorem we first give the following definition.

\begin{definition}
\label{defloclipset}
Let $X$ be a complex manifold. A subset $M \subseteq X$ is called a \emph{totally real Lipschitz set} if for each point $p \in M$ there exists a holomorphic coordinate neighborhood $z \colon U \to V$ containing $p$ such that $z(U \cap M)$ is contained in the graph of a Lipschitz map $\psi \colon \R_x^n \cap V \to \R_y^n \cap V$, where $z = x + i y$, and with Lipschitz constant strictly less than one.
\end{definition}

\begin{proof}[Proof of Theorem 1.3]
By Corollary \ref{lipremovable}, transition maps extend across totally real Lipschitz sets. Hence it is enough to show that there is a discrete set of points $P \subset M$ such that for any point $x_0 \in M \setminus P$, there is an open neighborhood $U$ of $x_0$ with the property that $E$ is holomorphically trivial on $U \setminus M$.

Let $\Omega \subset M$ be the set consisting of all points $x \in M$ that have a neighborhood $U \subset X$ such that $E|_{U \setminus M}$ is trivial. We will first show that $\Omega \neq \emptyset$, and then that every point $x \in b\Omega$ is isolated in $M$.

By Corollary \ref{densehpc}, there exists an embedding $\phi \colon \triangle^n \to X$ with $\phi(\triangle^n)\cap M \neq \emptyset$, and such that $\phi(\hartogs{1})\subset X\setminus M$.

By choosing a possibly smaller Hartogs figure, we may assume that $E$ is a vector bundle on $\phi(\hartogs{1})$. Now $E$ extends to a coherent analytic sheaf $\mathcal S$ on $\phi(\triangle^n)$, and $\mathcal S$ is locally free outside a finite set of points.   
Arguing as in the proof of Lemma \ref{extendsheaf}, we see that $E$ may be trivialized outside a locally finite set of points in $M$.

Next, $\tilde M := M \setminus\Omega$ is itself a totally real Lipschitz set, and $E$ extends to a vector bundle on $X\setminus\tilde M$. 
Unless $\tilde M$ is a discrete set of points, the whole argument can be repeated to show that $E$ is a trivial vector bundle near a relatively open subset of $\tilde M$, contradicting the maximality of $\Omega$.
\end{proof}

\section{Modification of a strictly plurisubharmonic function near a Morse critical point.}\label{modification}

In this section we show how to modify a strictly plurisubharmonic function near a nice Morse critical point, in order to extend sections and bundles across critical points.
This is very similar to what is done in Oka--Grauert theory, and the reader can compare the following proposition with Lemma 3.10.1 in \cites{Forstneric-book}.
We feel however that the proof we give here is considerably simpler.
It is motivated by ideas in \cites{FornaessSibonyWold}.

\begin{proposition}
\label{lemperturbflat}
Let $\rho \colon X \to \R$ be a strictly plurisubharmonic exhaustion function, let $x\in X$ be a nice critical point with $\rho(x)=c$ and let some $c'>c$ be such that there are no critical values in the interval $(c,c']$.
Choose coordinates $z$ in which $\rho$ assumes the form as in \eqref{nice} for $\|z\| < 1$. We can assume that all $\mu_j$ for $j=1,\dots,m$ in \eqref{nice} are positive (for if not, we may move the corresponding $x_j^2$ terms to the second sum and renumber).
Then for every $\epsilon>0$, there exist $\epsilon_1<\epsilon_2<\epsilon$ 
and $s > 0$, and there exists a strictly plurisubharmonic function $\tilde\rho$ on $X$ 
such that the following hold:
\begin{itemize}
\item[(i)] in the coordinate $z$, on $\mathbb B_{\epsilon_1}(0)$, we have that
\begin{equation}
\label{eqniceflat}
\tilde\rho(z)=c + \sum_{j=1}^{m} \left(x_j^2 + \mu_j y_j^2\right) + \sum_{j=m+1}^n x_j^2, 
\end{equation}
with $0 < \mu_j \leq 1$ for $j = 1,\dots,m,$
\item[(ii)] $\{x\in X:\tilde\rho(x)=c\}\setminus\mathbb B_{\epsilon_1}\subset\{x\in X \,:\, \rho(x)<c\}$, 
\item[(iii)] $\tilde\rho(x)=\rho(x)+s$ for $x\in X\setminus\mathbb B_{\epsilon_2}(0)$, 
and 
\item[(iv)] there are no critical values for $\tilde\rho$ in the interval $(c,c'+s]$.
\end{itemize}
\end{proposition}

\begin{proof}
For simplicity we assume that $c=0$, and hence we can locally represent the function $\rho$ as
\begin{equation}\label{local}
\rho(z)=\sum_{j=1}^m(x_j^2+\mu_jy_j^2) + \sum_{j=m+1}^n (x_j^2-\mu_jy_j^2).
\end{equation}
Fix a cutoff function $\chi \in \mathcal{C}^\infty_0(\R)$ with $\chi \equiv 1$ on $[-1/2,1/2]$, $\mathrm{supp}\, \chi \subset (-1, 1)$, 
and $0\leq\chi\leq 1$.
Furthermore, for every $0<\epsilon<1$, and every $\delta\in \{0\}^m \times [0,1)^{n-m}$, we set 
\begin{equation}
\psi_{\epsilon,\delta}(z) := \chi\left( \frac{\|z\|^2}{\epsilon^2} \right) \cdot \sum_{j=m+1}^n \delta_j y_j^2 
\end{equation}
The function $\tilde{\rho}$ shall be defined as a sum 
\begin{equation}\label{correctsum}
\tilde\rho(z) := \rho(z) + \sum_{k=1}^{\ell}\psi_{\epsilon_k,\delta(k)}(z) + s \cdot \left(1-\chi\left(\frac{\|z\|^2}{\epsilon_{\ell+1}^2}\right)\right).
\end{equation}
The following lemma provides the crucial step in the proof.

\begin{lemma}\label{intermediate}
Fix the function $\chi$ as above and let $0<\mu<1$ and $\mu'>0$.  Then there exists $0<\delta_0<1$ such that if $\rho$ is 
a function of the form \eqref{local} on $\mathbb B^n_{\epsilon}$ with $\mu_j\geq \mu'$ for $j=1,\dots,m$, and $\mu_j\leq\mu$ for $j=m+1,\dots,n$, 
and if $\psi_{\epsilon,\delta}$ as above satisfies $0<\delta_j\leq \mu_j$ and $\delta_j<\delta_0$ for 
$j=m+1,\dots,n$, then $\rho_{\epsilon,\delta}=\rho+\psi_{\epsilon,\delta}$ is a strictly plurisubharmonic function on $\mathbb B^n_\epsilon$, 
with all critical points contained in $\{\rho_{\epsilon,\delta}\leq 0\}$.
\end{lemma}

Before we prove the lemma, we show how it is used by completing the proof of the above proposition:

Let $\mu:=\max_{m+1\leq j\leq n}\{\mu_j\}, \; \mu':=\min_{1\leq j\leq m}\{\mu_j\}$, and let $\delta_0$ be as in the lemma, depending on $\mu$ and $\mu'$, and fix an $l\in\mathbb N$ and $\delta_j$ for $j=m+1,\dots,n$ such that $\mu_j=l\delta_j$ with $\delta_j<\delta_0$.
Set $\epsilon_k=(1/2)^k$ for $k=1,\cdots,l+1$, and set $\delta(k)=(\delta_{m+1},\dots,\delta_n)$ for $k=1,\dots,l$.
The function $\tilde{\rho}$ defined by \eqref{correctsum} will now satisfy all the desired properties for sufficiently small $s$. 
\end{proof}

\begin{figure}
\begin{center}
\begin{pspicture}(-5,-4)(6,6)
\NormalCoor
\psline(-3,-3)(3,3)
\psline(-3,3)(3,-3)
\pscircle(0,0){.5}
\pscircle(0,0){1.6}
\psline[linewidth=2pt](0,-0.8)(0,0.8)
\psplot[linewidth=2pt]{0.7}{3}{x dup mul 1 add sqrt}
\psplot[linewidth=2pt]{-3}{-0.7}{x dup mul 1 add sqrt}
\psplot[linewidth=2pt]{0.7}{3}{x dup mul 1 add sqrt neg}
\psplot[linewidth=2pt]{-3}{-0.7}{x dup mul 1 add sqrt neg}
\pscurve[linewidth=2pt](0,0.8)(0.1,1)(0.3,0.8)(0.5,1)(0.7,1.2)(0.8,1.3)
\pscurve[linewidth=2pt](0,0.8)(-0.1,1)(-0.3,1.1)(-0.5,1)(-0.7,1.2)(-0.8,1.3)
\pscurve[linewidth=2pt](0,-0.8)(0.1,-1)(0.4,-0.9)(0.5,-1.1)(0.7,-1.2)(0.8,-1.3)
\pscurve[linewidth=2pt](0,-0.8)(-0.1,-1)(-0.3,-1.1)(-0.4,-1)(-0.7,-1.2)(-0.8,-1.3)
\rput(2.5,0.5){$\rho > 0$}
\rput(-2.5,0.5){$\rho > 0$}
\rput(0.7,2.5){$\rho < 0$}
\rput(0.7,-2.5){$\rho < 0$}
\rput(-1.5,3){$\tilde\rho = 0$}
\psframe(-2.1,3.3)(-0.7,2.8)
\pscurve[linewidth=.5pt]{->}(-1.5,2.8)(-1.2,2.3)(-1,1.5)
\rput(1.1,0){$\mathbb{B}_{\epsilon_2}(0)$}
\end{pspicture}
\end{center}
\caption{\label{img-criticalflat}} Modification of $\rho$ to $\tilde\rho$ in Proposition \ref{intermediate}, assuming for simplicity that the form \eqref{eqniceflat} is reached at this step.

\end{figure}

\begin{proof}[Proof of Lemma \ref{intermediate}]
By computing the Levi form of $\psi_{\epsilon, \delta}$, we see that there exists a constant $M$, independent of $\epsilon$ and $\delta$, such that \\ $\mathcal L_z(\psi_{\epsilon,\delta};w) \geq -M \delta \|w\|^2$, where  $\delta = \max\limits_{m < j \leq n} \{ \delta_j \}$, hence 
\[
\mathcal L_z(\rho + \psi_{\epsilon, \delta};w) \geq \frac{1}{2} \left(\sum_{j=1}^m (1 + \mu_j - M \delta) |w_j|^2 + \sum_{j=m+1}^n (1 - \mu_j - M \delta) |w_j|^2\right).
\]
So if $\delta_0<\frac{1-\mu}{M}$ we see that $\rho_{\epsilon,\delta}$ is strictly plurisubharmonic.  

Next we compute the gradient  
\begin{align*}
\nabla\rho_{\epsilon,\delta}(z) & = \nabla\rho(z) + \nabla\psi_{\epsilon,\delta}(z) \\
& = 2(x_1, \dots, x_n, \mu_1y_1, \cdots,\mu_my_m,-\mu_{m+1}y_{m+1}, \dots, -\mu_n y_n) \\
& + \chi\left(\frac{\|z\|^2}{\epsilon^2}\right) \cdot 2(0, \dots, 0, \delta_{m+1} y_{m+1}, \dots, \delta_n y_n)\\
& + \left(\frac{2}{\epsilon^2}\chi'\left(\frac{\|z\|^2}{\epsilon^2}\right) \cdot \sum_{j=m+1}^n \delta_j y_j^2\right) \cdot (x_1, \dots ,x_n, y_1, \dots, y_n).
\end{align*}
Now fix $z$ and assume that $\nabla \rho_{\epsilon,\delta}(z)=0$. 
If $y_{m+1} = y_{m+2} = \dots = y_n =0$, then $\nabla\rho_{\epsilon,\delta}(z)=\nabla\rho(z)$, hence $x_1 = x_2 = \dots = x_n = y_1= y_2 = \cdots = y_m = 0$.

Next we assume that $x \neq 0$ and observe that this requires \\ $\frac{2}{\epsilon^2}\chi'\left(\frac{\|z\|^2}{\epsilon^2}\right) \cdot \sum_{j=1}^n \delta_j y_j^2 = -2$.
Now there is at least one $j \in \{m+1, \dots, n\}$ such that $y_j \neq 0$. The $y_j$-component in the sum of the first two lines above is $2 y_j \left(\chi\left(\frac{\|z\|^2}{\epsilon^2}\right) \delta_j - \mu_j\right)$.
But the expression in brackets can never be $1$, so the $y_j$-component  of $\nabla\rho_{\epsilon,\delta}(z)$ is not zero.
Finally, assume that $y_k\neq 0$ for $1\leq k\leq m$.
Note that $\sum_{j=m+1}^n \delta_j y_j^2\leq \delta_0\|z\|^2<\delta_0\epsilon^2$, so 
the $y_k$-component of the third term in the sum above is less than $ 2 \|\chi'\|\delta_0|y_k|$ in absolute value.
Therefore, as long as $\delta_0$ is chosen small enough so that $2 \|\chi'\|\delta_0<\mu'$, the gradient is non-zero.

We conclude that 
$x_1 = x_2 = \dots = x_n = y_1 = \cdots =y_m = 0$, and $\rho_{\epsilon,\delta}(z) \leq 0$.
\end{proof}

\begin{proof}[Proof of Theorem \ref{thm1}]
We choose a nice nonnegative plurisubharmonic exhaustion function $\rho$ of $X$, i.e.\ as in Definition \ref{defnice} with $L=\emptyset$. Then, by the compactness of $K$, the vector bundle $E$ is defined on a superlevel set $X^a \subset X \setminus K$ of $\rho$ for some regular value $a > 0$.
Let 
$
\displaystyle
c := \inf \{ b < a \;:\; E \mbox{ extends to } X^{b} \setminus Q_{b}, \mbox{ where } Q_{b} \mbox{ is a discrete closed subset of } X^b \mbox{ which is } \\ \mbox{ the singularity set of } E \mbox{ in } X^b \}.
$
It follows from Corollary \ref{inductivelimit} that $E$ extends to $X^c \setminus Q_c$ where $Q_c$ is also is a discrete closed subset of $X^c$ which is the singularity set of $E$ in $X^c$.
By Corollary \ref{noncriticalextension}, $c$ cannot be a regular value of $\rho$.
If the critical point is not a local minimum, we can consider the function $\tilde\rho$  constructed in Proposition \ref{lemperturbflat} and conclude that $E$ extends to $\{ \tilde\rho > c \} \setminus \tilde Q_c$.
By Lemma \ref{smoothgraph} and an argument as in the Lemma \ref{extendsheaf} the set $\tilde Q_c$ is locally finite near $M := \{ \tilde \rho = 0 \} \cap \mathbb{B}_{\epsilon_1}$. Finally, $E$ extends across $M$ by Theorem \ref{thm3}.
If the critical point is a local minimum, it is easy to find a Hartogs figure, and by the same argument as before, the singularity set cannot accumulate there. We conclude that $c=0$ and that the singularity set is finite.
\end{proof}

\section{Oka theory}
\label{sec-oka}

In this section we will use our extension result to prove the following theorem.

\begin{theorem}\label{finitesetofpoints}
Let $X$ be a Stein manifold with an exhaustion function $\rho$, let $K \subset X$ be a compact set, and let $E_j \to X \setminus K$ be a holomorphic vector bundle with sprays $s_j \colon E \to X \setminus K$ for $j = 1, \dots, m$ finite. Assume further that the family of sprays is dominating on a super-level set $X^c \supset K$.
\begin{enumerate}
\item If $m=1$, then $K$ has only finitely many accumulation points which lie in $X_c$.
\item If $m \geq 1$ and $K$ is smoothly bounded, then $K$ is actually the empty set.
\end{enumerate}
\end{theorem}

The statement of Theorem \ref{finitesetofpoints} contains Theorem \ref{thm2}.

Before proving this theorem we first give some background on Oka theory.

Gromov \cites{Gromov-elliptic} introduced 1989 in his seminal paper the notion of an elliptic manifold and proved an Oka principle for holomorphic sections of elliptic bundles, generalizing previous work of Grauert \cites{Grauert-Lie} for complex Lie groups. The (basic) Oka principle of a complex manifold $X$ says that every continuous map $Y \to X$ from a Stein space $Y$ is homotopic to a holomorphic map.

In Oka theory exist many interesting classes of complex manifolds with weaker properties than ellipticity. However, all the known inclusion relations between these classes are yet not known to be proper inclusions. Following the book of Forstneri\v{c} \cite{Forstneric-book}*{Chap.\ 5} we want to mention in particular the following classes and then prove that at least one of these inclusions has to be proper.

\begin{definition}
A \emph{spray} on a complex manifold $X$ is a triple $(E, \pi, s)$ consisting of a holomorphic vector bundle $\pi \colon E \to X$ and a holomorphic map $s \colon E \to X$  such that for each point $x \in X$ we have $s(0_x) = x$ where $0_x$ denotes the zero in the fibre over $x$.

The spray $(E, \pi, s)$ is said to be \emph{dominating} if for every point $x \in X$ we have
\[
\mathrm{d}_{0_x} s (E_{x}) = \mathrm{T}_x X
\]
A complex manifold is called \emph{elliptic} if it admits a dominating spray.
\end{definition}

In this definition we adapted the convention used e.g. in the textbook \cites{Forstneric-book} identifying the fibre $E_x$ over $x$ with its tangent space in $0_x$.

A weaker notion, subellipticity, was introduced later by Forstneri\v{c} \cites{Forstneric-subelliptic} where he proved the Oka principle for subelliptic manifolds:

\begin{definition}
A finite family of sprays $(E_j, \pi_j, s_j), \; j = 1, \dots, m,$ on $X$
is called \emph{dominating} if for every point $x \in X$ we have
\begin{equation}\label{dominating}
\mathrm{d}_{0_x} s_1 (E_{1,x}) + \dots + \mathrm{d}_{0_x} s_m (E_{m,x}) = \mathrm{T}_x X
\end{equation}
A complex manifold $X$ is called \emph{subelliptic} if it admits a finite dominating family of sprays.
\end{definition}

It is immediately clear from the definition that an elliptic manifold is subelliptic. However, the Oka principle holds under even weaker conditions. Forstneri\v{c} \citelist{\cite{Forstneric-Oka1} \cite{Forstneric-Oka2} \cite{Forstneric-Oka3}} showed that the following condition (CAP) is equivalent for a manifold to satisfy the Oka principle (and versions of the Oka principle with interpolation and approximation), hence justifing the name Oka manifold:

\begin{definition}
A complex manifold $X$ is said to satisfy the \emph{convex approximation property} (CAP) if on any compact convex set $K \subset \C^n, \; n \in \N,$ every holomorphic map $f \colon K \to X$ can be approximated, uniformly on $K$, by entire holomorphic maps $\C^n \to X$. 
A manifold satisfying CAP is called an \emph{Oka manifold}.
If this approximation property holds only for $n \leq N$ for some $N \in \N$, then $X$ is said to satisfy $\mathrm{CAP}_N$.
\end{definition}

A subelliptic manifold is always Oka. Whether an Oka manifold is elliptic or subelliptic, is on the other hand not known -- this implication holds however under the extra assumption that it is Stein.

We want to mention also these weaker properties:

\begin{definition}
A complex manifold $X$ of dimension $n$ is called \emph{dominable} if there exists a point $x_0 \in X$ and a holomorphic map $f \colon \C^n \to X$ with $f(0) = x_0$ and $\mathrm{rank}\, \mathrm{d}_0 f = n$.
\end{definition}

\begin{definition}
A complex manifold $X$ of dimension $n$ is called \emph{strongly dominable} if for every  point $x_0 \in X$ there exists a holomorphic map $f \colon \C^n \to X$ with $f(0) = x_0$ and $\mathrm{rank}\, \mathrm{d}_0 f = n$.
\end{definition}

Summarizing the previous, the following inclusions are known:
\[
\begin{split}
\mathrm{elliptic} \subseteq \mathrm{subelliptic} 
 \subseteq \mathrm{Oka} \subseteq \mathrm{strongly \, dominable} \subseteq \mathrm{dominable}
\end{split}
\]

There are several known candidates to prove that one of these inclusions is proper.
Theorem \ref{finitesetofpoints} gives a class of examples of complex manifolds which are not subelliptic, but strongly dominable.

\begin{corollary}
\label{cor-oka-ex}
Let $K$ be the closure of a smoothly bounded non-empty open subset of $\C^n, \, n \geq 3$. Then $\C^n \setminus K$ is not subelliptic, but dominable. If $K$ is holomorphically convex, then $\C^n \setminus K$ is strongly dominable. 
\end{corollary}
\begin{proof}
Since $X := \C^n$ is a Stein manifold, it is a direct consequence of Theorem \ref{finitesetofpoints} that $X \setminus K$ is not subelliptic. However for all $ x \in \C^n \setminus \widehat{K}$ there exists a Fatou--Bieberbach domain $\Omega \subseteq \C^n \setminus \widehat{K}, \Omega \cong \C^n, x \in \Omega$, therefore $\C^n \setminus K$ is dominable resp.\ strongly dominable.
\end{proof}

Whether or not these examples are Oka manifolds, remains an open question. A partial result in this direction has been obtained in \cites{ForstnericRitter} where they prove that $\C^n\setminus\overline{\mathbb B^n}$ satisfies $\mathrm{CAP}_k$ for $k<n$.

\begin{proof}[Proof of Theorem \ref{finitesetofpoints}]
In the subelliptic case (which covers also the elliptic case) we are given a finite family of sprays $(E_j,\pi_j,s_j)$ for $j=1,...,m$ on $X\setminus K$, which is dominating on $X^c$.
By Theorem \ref{thm1} we can extend it to a family $(\tilde E_j,\tilde\pi_j,\tilde s_j)$ of sprays on $X\setminus P$, $\tilde s_j \colon \tilde E_j \to X$, where $P$ is a finite set of points.   
Now $(\tilde E_j,\tilde\pi_j,\tilde s_j)$ is dominating outside an analytic subset of $X\setminus P$, hence a closed discrete set of points $P^\prime$ in $X \setminus P$. We denote by $P^{\prime\prime} := P \cup P^\prime$.
Consider a point $p \in b K \setminus P^{\prime\prime}$, i.e.\ a boundary point of $K$ where the family of the extended sprays exists and is dominating. By continuity of the derivative and lower semicontiuity of its rank, the family of sprays is then also dominating in a open neighborhood $V \subset X \setminus P^{\prime\prime}$ of $p$.
\begin{enumerate}
\item If $m=1$, every sequence in $V \setminus K$ converging to $p$ will be such that the fibres above all but finitely many points of this sequence hit $p$. For if not, $s_1$ could not be dominating in $p$ by the continuity of the derivative. Therefore $K \subseteq P^{\prime\prime}$, since we have just shown that there cannot exist a dominating spray $\tilde s_1 \colon E_1|_{X \setminus K} \to X \setminus K$ if $K \not\subseteq P^{\prime\prime}$.
\item If $m \geq 1$, assume for a contradiction that $K \neq \emptyset$ is smoothly bounded. For each spray $\tilde s_j$ we denote by $A_{p,j} := \{ q \in V \,:\, \tilde s_j(E_{j,q}) \ni p \} = \pi_j((\left.\tilde s_j\right|_V)^{-1}(p))$ the set of points whose fibre hits $p$. By the domination property in $p$ there exists at  least one spray $\tilde s_j$ such that $A_{p,j}$ has a tangent in $p$ which is transversal to the boundary of $K$. In analogue to the situation $m=1$ we can choose a sequence in $A_{p,j} \setminus K$ converging to $p$ such that the fibres above the points of this sequence hit $p$ under $\tilde s_j$. Again we can conclude $K \subseteq P^{\prime\prime}$. Now if the boundary of $K$ is actually smooth and consists of only finitely many accumulation points, $K$ is empty.
\qedhere
\end{enumerate} 
\end{proof}


\end{document}